\documentclass[11pt,oneside,a4paper]{article}
\usepackage{syntonly}
\usepackage{geometry}
\usepackage{mathrsfs}
\usepackage{makeidx}
\makeindex
\usepackage{amsmath}
\usepackage{amssymb}
\usepackage{amsmath}
\usepackage{amsfonts}
\usepackage{amssymb}
\usepackage{esvect}
\usepackage{algorithm}
\usepackage{bm}
\usepackage{tikz}
\usepackage{cleveref}

\usepackage{graphicx}
\usepackage{epstopdf}

\newtheorem{theorem}{Theorem} 
\newtheorem{lemma}{Lemma}

\newtheorem{corollary}{Corollary}

\newenvironment{proof}[1][Proof]{\begin{trivlist}
\item[\hskip \labelsep {\bfseries #1}]}{\end{trivlist}}

\newcommand{\qed}{\nobreak \ifvmode \relax \else
      \ifdim\lastskip<1.5em \hskip-\lastskip
      \hskip1.5em plus0em minus0.5em \fi \nobreak
      \vrule height0.30em width0.4em depth0.25em\fi}

 \author{Safari Mukeru\\
\footnotesize{\em Department of Decision Sciences}\\ 
\footnotesize{University of South Africa, P. O. Box 392, Pretoria, 0003. South Africa}\\
\footnotesize{e-mail: mukers@unisa.ac.za}}

\title{{Some applications of the Menshov--Rademacher theorem}}

\date{}

\begin{document}

\maketitle

\pagenumbering{arabic}

\begin{abstract}
Given a sequence $(X_n)$ of real or complex random variables and a sequence of numbers $(a_n)$, an interesting  problem is to determine the conditions under which  the series $\sum_{n=1}^\infty a_n X_n$ is almost surely convergent. This paper extends the classical Menshov--Rademacher theorem on the convergence of orthogonal series  to general series of dependent random variables and derives interesting sufficient conditions for the almost everywhere convergence of trigonometric series with respect to singular measures  whose Fourier transform decays to 0 at infinity with positive rate.



\end{abstract}

\vspace*{2.5mm}

\section{Introduction}
A classical fundamental result obtained independently by Menshov and Rademacher in the 1920's states that if the condition $\sum_{n=1}^\infty |a_n|^2 \log_2^2(n+1) < \infty$ is satisfied for a given sequence of real or complex numbers $(a_n)$, then for any sequence of orthonormal sequence $(\varphi_n)$ of $L^2$ functions on the unit circle $\mathbb{T}$, the series $\sum_{k=1}^\infty a_n \varphi(x)$ converges almost everywhere on  $\mathbb{T}$ (with respect to the Lebesgue measure). This condition has been extended to the $L^2$ space of general measure spaces and more recently Paszkiewicz \cite{Paszkiewicz1, Paszkiewicz2} and Bednorz \cite{Bednorz2} obtained a  necessary and sufficient condition on the sequence $(a_n)$ for the convergence a.e. of  $\sum_{k=1}^\infty a_n \varphi(x)$ (the condition is expressed in terms of the existence of a majorising measure on the set $\{\sum_{k = 1}^n a_k^2: n \geq 1\} \cup\{0\}$).  Menshov and Rademacher theorem was probably the first result related to random series of {\it non-independent} random variables. The problem of convergence of random series with dependent random variables and the problem of extending limit theorems of probability theory related to iid variables to dependent random variables has attracted a lot of attention since the introduction of sequences of negatively associated random variables  by Joag-Dev and Proschan \cite{Joag-Dev}. However the problem remains largely open. 

The problem of dependent random variables is very wide and it is difficult to find a general criteria that applies to all. In the literature the general approach is to delineate certain forms of dependency and classes of random variables and determine an appropriate condition for the a.s. convergence. {Matu\l{}a} \cite{Matula} proved that if $(X_n)$ are negatively associated real random variables with finite second moments then the condition $\sum_{n=1}^\infty \mathbb{E}(|X_n|^2) < \infty$ implies the a.s. convergence of $\sum_{n=1}^\infty X_n$. This has been extended to random vectors in a Hilbert space by Ko, Kim and Han \cite{Ko}.

Antonini, Kozachenko and Volodin \cite{Antonini} considered the case of sub-Gaussian random variables $(X_n)$ exhibiting certain dependence structures (negative association, $m$-dependence and $m$-acceptability) and derived interesting sufficient conditions for the a.s. convergence.  It is shown in \cite{Mukeru_JOTP_2020} that if $(X_n)$ can be expressed as linear combinations $X_n = \sum_{k=1}^n a_{n,k} Z_k$ for a fixed  sequence $(Z_n)$ of i.i.d random variables (with zero mean and unit variance) and complex numbers $(a_{n,k})$, then the condition  $\sum_{n=1}^\infty \left(\sum_{k=1}^\infty |a_{n+k-1,k}|^2\right)^{1/2} < \infty$ is sufficient for the a.s. convergence of the series $\sum_{n=1}^\infty X_n$. (For the problem of  limit theorems of sequences of dependent random variables we refer to Naderi et al. \cite{Naderi} and references therein for some recent developments.)  

In this paper we consider the problem of convergence of the random series $\sum_{n=1}^\infty a_n X_n$ where $(a_n)$ is a sequence of  real or complex numbers and $(X_n)$ a sequence of real or complex random variables in the most general dependence case. Our approach consists of extending the  Menshov--Rademacher theorem from the classical case of orthonormal series  to general dependent random variables. We thus obtain an explicit sufficient condition for the almost sure convergence of random series in their most generality. This condition is also applied to study almost everywhere convergence of trigonometric Fourier series on some subsets of the unit circle of Lebesgue measure zero. 

Our results can be summarised are as follows:
If $\sum_{m,n=1}^\infty |a_n| |a_m| |\mathbb{E}(X_n \overline{X_m})| \log_2(n+1) \log_2(m+1) < \infty$, then the series $\sum_{n =1}^\infty a_n X_n$ converges a.s. From this result, we derive the following: If the linear operator defined by the matrix $A = (|\mathbb{E}(X_n \overline{X_m})|)$ in $\ell^2$ is bounded, then Menshov--Rademacher condition $\sum_{n=1}^\infty |a_n|^2 \log_2^2(n+1) < \infty$ for the a.s. convergence of orthonormal series is  also sufficient for the a.s. convergence of the series $\sum_{n=1}^\infty a_n X_n.$ The same is true if $(a_n)$ are real numbers  and $(X_n)$ are real random variables such that $\mathbb{E}(X_n X_m) \leq 0$ for all $m \ne n$. If the matrix $(n^{-b} m^{-b} A_{n,m})$ is bounded for some $b\geq 0$, then the condition $\sum_{n=1}^\infty |a_n|^2 n^{2b} \log_2^2(n+1) < \infty$ is sufficient for the a.s. convergence of $\sum_{n=1}^\infty a_n X_n$.  Finally if $\mu$ is a Borel probability measure on the unit circle $\mathbb{T}$ such that its Fourier transform satisfies $|\hat \mu(n)| \leq K |n|^{-a}$ for some $a \geq 0$ and there exists $b > (1-a)/2$ such that $\sum_{n\in \mathbb{Z}} |a_n|^2 |n|^{2b} \log_2^2(n+1) < \infty$, then the Fourier trigonometric series $\sum_{n \in \mathbb{Z}} a_n \exp(2\pi i n t)$ converges $\mu$-almost everywhere on $\mathbb{T}$. This implies that the Fourier series of a well-behaved $L^2$ function $f$ on $\mathbb{T}$ in the sense that $\sum_{n=1}^\infty |\hat f(n)|^2 n^{2b} \log_2^2(n+1) < \infty$ for some number $0 \leq b \leq 1/2$ cannot diverge everywhere on a subset of Fourier dimension $> 1 - 2b$. In particular if a function belongs to the Sobolev space $H^p(\mathbb{T})$ for some $0 \leq p \leq 1/2$, then its Fourier series converges $\mu$-almost everywhere on every subset $E$ of $\mathbb{T}$ of Fourier dimension $> 1 - 2p$ where $\mu$ is a Borel measure supported by $E$ such that $|\mu(n)| = o(|n|^{-(1-2p)}).$

The paper concludes with a remark concerning the particular case of Gaussian random variables were we prove (using the classical Sudakov--Fernique inequality) that the condition $\sum_{m,n=1}^\infty |a_n| |a_m| |\mathbb{E}(X_n \overline{X_m})| < \infty$ is enough for the a.s. convergence of $\sum_{n=1}^\infty a_n X_n$.

\section{An extension of Menshov--Rademacher theorem}
For a sequence $(a_n)$ of numbers the classical Menshov--Rademacher theorem says that if the condition 
\begin{eqnarray}
\sum_{n=1}^\infty |a_n|^2\log_2^2(n+1) < \infty,
\end{eqnarray}
then the series $\sum_{n=1}^\infty a_n \varphi(x)$ converges almost everywhere for any sequence $(\varphi_n)$ of  orthonormal functions in $L^2(\mathbb{T})$. (Here $\mathbb{T}$ is the unit circle $\mathbb{R}/\mathbb{Z}$.) In this paper we relax the condition of orthogonality and extend Menshov--Rademacher theorem to series $\sum_{n=1}^\infty a_n \varphi(x)$ in that general case.  
 In general we consider random variables in the $L^2(\Omega)$ space of a certain probability space.  The usual  inner product in  $L^2(\Omega)$  is
$\langle X,Y\rangle = \mathbb{E}(X \overline{Y}).$ 

\begin{theorem}\label{Nderhe}
For any sequence of real or complex numbers $(a_n)$ and a sequence $(X_n)$ of real or complex random variables such that $\mathbb{E}(|X_n|^2) = 1$ for all $n$, if  
 $$L:= \sum_{n,m=1}^\infty |a_n| |a_m| | \mathbb{E}(X_n \overline{X_m})| \log_2(n+1) \log_2(m+1) < \infty,$$ then the series 
$\sum_{n=1}^\infty a_n X_n$ converges almost surely. Moreover, 
$$\mathbb{E}\left(\sup_{n\in \mathbb{N}}|a_1 X_1 + a_2 X_2 + \ldots + a_n X_n|^2\right) \leq 8 L.$$
\end{theorem}
The proof  is based on the classical proof of the Menshov--Rademacher theorem as given in 
Kashin and Saakyan \cite[p 251]{Kashin_Saakyan}. We have also used a proof given by Mikhailets and Murach \cite{Mikhailets}.  The following lemma is  also an extension of the classical Menshov - Rademacher lemma. 

\begin{lemma}
With the hypothesis of the theorem we have the following inequality: For any integer $N >0$, let
 $$S_N^* = \max_{1 \leq j \leq N} |\sum_{n=1}^j a_n X_n|.$$ Then 
$$\mathbb{E}\left(S_N^*\right)^2 \leq (2 + \log_2 N)^2 \sum_{n, m=1}^N |a_n| |a_m| |\mathbb{E}(X_n \overline{X_m})|.$$
\end{lemma}
\begin{proof}
The starting point is to assume that $N = 2^r$ for some integer $r\geq 1$ and show that 
$$\mathbb{E}\left(S_N^*\right)^2 \leq (1 + \log_2 N)^2 \sum_{n, m=1}^N |a_n| |a_m| |\mathbb{E}(X_n \overline{X_m})|.$$
The case where  $2^{r-1} < N \leq 2^r$ is reduced to the previous one by just taking 
$a_n = 0$ for all $n$ such that $N < n \leq 2^r$. Then this inequality holds for $r - 1 < \log_n N$, that is,
$$\mathbb{E}\left(S_N^*\right)^2 \leq (2 + \log_2 N)^2 \sum_{n, m=1}^N |a_n| |a_m| |\mathbb{E}(X_n \overline{X_m})|.$$
Now for $N = 2^r$, consider for each number $j \in \{1,2, \ldots, 2^r\}$ its dyadic representation
 $$j = \sum_{k=0}^r \xi_k(j) 2^{r-k},\,\,\mbox{ with } \xi_k(j)  = 0 \mbox{ or } 1, \mbox{ for all } k.$$
 Moreover for each such $j$, decompose the interval $[0, j]$ in $\mathbb{N}$ into subintervals
 $$[1, j] =  \bigcup_{k: \xi_k(j) \ne 0} I_k$$ where 
  $$I_k := I_k(j) = \left\{n \in \mathbb{N}: \sum_{s=0}^{k-1} \xi_s(j)2^{r-s} < n  \leq \sum_{s=0}^{k} \xi_s(j)2^{r-s}\right\}.$$
Now write
$$   \sum_{n=1}^j a_n X_n = \sum_{k: \xi_k(j) \ne 0} \sum_{n \in I_k(j)} a_n X_n.$$
Then
  \begin{eqnarray*}
   |\sum_{n=1}^j a_n X_n| &  =  & |\sum_{k: \xi_k(j)\ne 0} \sum_{n \in I_k} a_n X_n|\\
   & \leq & \sum_{k: \xi_k(j) \ne 0} 1. |\sum_{n \in I_k} a_n X_n|\\
   & \leq & \left(\sum_{k: \xi_k(j) \ne 0} 1^2\right)^{1/2}. \left(\sum_{k: \xi_k(j) \ne 0} |\sum_{n \in I_k(j)} a_n X_n|^2\right)^{1/2}\\
     \end{eqnarray*}
 Clearly $$\sum_{k: \xi_k(j) \ne 0} 1 \leq \sum_{k= 0}^r 1 = (1+r)$$ and 
for all $j$,   
     $$\sum_{k: \xi_k(j) \ne 0} |\sum_{n \in I_k(j)} a_n X_n|^2 \leq \sum_{k=0}^r \sum_{p=0}^{2^{k} - 1} \left |\sum_{n = p 2^{r-k} + 1}^{(p+1) 2^{r-k}} a_n X_n\right|^2.$$
 Hence simultaneously for all $1 \leq j \leq N$, 
   \begin{eqnarray*}
   |\sum_{n=1}^j a_n X_n|^2 \leq (r+1) \sum_{k=0}^r \sum_{p=0}^{2^{k} - 1} \left |\sum_{n = p 2^{r-k} + 1}^{(p+1) 2^{r-k}} a_n X_n\right|^2.
   \end{eqnarray*}
Now assume 
    $$S_N^* = |\sum_{n=1}^{j} a_n X_n|, \,\,\,\, j \mbox{ random }.$$        
Then 
  $$(S_N^*)^2  \leq (r+1) \sum_{k=0}^r \sum_{p=0}^{2^{k} - 1} \left |\sum_{n = p 2^{r-k} + 1}^{(p+1) 2^{r-k}} a_n X_n\right|^2$$ and hence
 \begin{eqnarray*}
\mathbb{E}(S_N^*)^2 &\leq & (r+1)\sum_{k=0}^r \sum_{p=0}^{2^{k} - 1} \mathbb{E}\left |\sum_{n = p 2^{r-k} + 1}^{(p+1) 2^{r-k}} a_n X_n\right|^2\\
  &  =  & (r+1)\sum_{k=0}^r \sum_{p=0}^{2^{k} - 1} \sum_{n,m = p 2^{r-k} + 1}^{(p+1) 2^{r-k}} a_n \overline{a_m}\, \mathbb{E}(X_n \overline{X_m})\\
  & \leq & (r+1)\sum_{k=0}^r \sum_{p=0}^{2^{k} - 1} \sum_{n,m = p 2^{r-k} + 1}^{(p+1) 2^{r-k}} |a_n| |a_m| | \mathbb{E}(X_n \overline{X_m})|\\
  & \leq & (r+1)\sum_{k=0}^r \sum_{n,m=0}^{2^r} |a_n| |a_m| | \mathbb{E}(X_n \overline{X_m})|\\
  & \leq & (r+1)^2 \sum_{n, m = 1}^N |a_n| |a_m| |\mathbb{E}(X_n \overline{X_m})|.
\end{eqnarray*}        
 Therefore,
  \begin{eqnarray*}
   \mathbb{E}(S_N^*)^2 &\leq & (1 + \log_2 N)^2 \sum_{n, m = 1}^N |a_n| |a_m| |\mathbb{E}(X_n \overline{X_m})|.
   \end{eqnarray*}
\end{proof}

\section{Proof of Theorem \ref{Nderhe}}   
 The first step is to consider the sequence of partial sums  of the form
  $$S_{2^k} = \sum_{n=1}^{2^k} a_n X_n,\,\,\, \mbox{ for } k = 1,2,3,\ldots $$
and show that $(S_{2^k})$ converges almost surely and write
$$S^* = \sup_{0 \leq k < \infty} |S_{2^k}|.$$
Set 
   $$\chi_k = \sum_{n=2^k}^{2^{k+1}-1} a_n X_n, k = 0,1,2,3, \ldots$$ and then 
     $$\mathbb{E}(|\chi_k|^2) \leq \sum_{n,m=2^{k}}^{2^{k+1}-1}|a_n||a_m| |\mathbb{E}(X_n \overline{X_m})|.$$
   Then,
\begin{eqnarray*}
\sum_{k=0}^\infty \mathbb{E}(|\chi_k|^2) (k+1)^2 &\leq & \sum_{k=0}^\infty (k+1)^2 \sum_{n,m=2^{k}}^{2^{k+1}-1}|a_n||a_m| |\mathbb{E}(X_n \overline{X_m})|\\
& \leq & \sum_{k=0}^\infty \sum_{n,m=2^{k}}^{2^{k+1}-1}|a_n||a_m| |\mathbb{E}(X_n \overline{X_m})| (1+\log_2 n)((1+\log_2 m) \\
\end{eqnarray*}
because for $k$ fixed and $n \geq 2^{k}$, then $(1 + k) \leq (1 + \log_2 n)$ and similarly, $(1 + k) \leq (1 + \log_2 m)$.  Moreover since clearly $(1 + \log_2 n) \leq \sqrt{2} \log_2(n+1)$, it follows that
\begin{eqnarray*}
\sum_{k=0}^\infty \mathbb{E}(|\chi_k|^2) (k+1)^2 \leq 2 \sum_{k=0}^\infty \sum_{n,m=2^{k}}^{2^{k+1}-1}|a_n||a_m| |\mathbb{E}(X_n \overline{X_m})| \log_2 (n+1)\log_2(m + 1).
 \end{eqnarray*}
Hence
$$\sum_{k=0}^\infty \mathbb{E}(|\chi_k|^2) (k+1)^2 \leq 2 L < \infty.$$
Moreover by the Cauchy-Schwarz inequality,
\begin{eqnarray*}
\sum_{k=0}^\infty (\mathbb{E}(|\chi_k|^2))^{1/2}  & = & \sum_{k=0}^\infty (\mathbb{E}(|\chi_k|^2))^{1/2} (k+1)(k+1)^{-1} \\
& \leq & \left(\sum_{k=0}^\infty \mathbb{E}(|\chi_k|^2) (k+1)^2\right)^{1/2} \left(\sum_{k=0}^\infty (k+1)^{-2}\right)^{1/2} \leq 2 \sqrt{L}.
\end{eqnarray*}           
This implies
$$\sum_{k=0}^\infty \mathbb{E}(|\chi_k|) \leq  \sum_{k=0}^\infty (\mathbb{E}(|\chi_k|^2))^{1/2} \leq 2 \sqrt{L}$$  and hence Beppo Levi theorem implies that 
 $\sum_{k=0}^\infty |\chi_k| < \infty$ almost surely. Hence the sequence $(S_{2^k})$ is a Cauchy sequence and therefore it converges almost surely. Moreover,
    $$\mathbb{E}((S^*)^2) \leq \sum_{k=0}^\infty \mathbb{E}(|\chi_k|^2) \leq 4 L.$$
Let us now consider 
 $$S^{\circ}_k  = \max_{2^{k} \leq j <  2^{k+1}}|\sum_{n=2^k}^j a_n X_n|,\,\, k = 1,2,3, \ldots\,\,\,  \mbox{  and } S^{\circ} =  \sup_{1 \leq k < \infty} S^{\circ}_k.$$ 
 By the lemma, 
  $$\mathbb{E}((S^{\circ}_k)^{^2})  \leq (2 + \log_2(j - 2^k + 1))^2 \sum_{n,m=2^k}^j |a_n| |a_m| |\mathbb{E}(X_n \overline{X_m})|. $$
 Then 
 \begin{eqnarray*}
 \sum_{k=0}^\infty \mathbb{E}((S^{\circ}_k)^{^2}) &\leq  &\sum_{k=0}^\infty (2 + \log_2 2^k)^2 \sum_{n,m=2^k}^{2^{k+1}-1} |a_n| |a_m| |\mathbb{E}(X_n \overline{X_m})|\\
 &\leq & \sum_{k=0}^\infty (2 + \log_2 2^k)^2 \sum_{n,m=2^k}^{2^{k+1}-1} |a_n| |a_m| |\mathbb{E}(X_n \overline{X_m})|\\
 & \leq &\sum_{k=0}^\infty  \sum_{n,m=2^k}^{2^{k+1}-1} |a_n| |a_m| |\mathbb{E}(X_n \overline{X_m})|(2 + \log_2 n)(2 + \log_2 m) \\
 & \leq & \sum_{n,m=1}^{\infty} |a_n| |a_m| |\mathbb{E}(X_n \overline{X_m})|(2 + \log_2 n)(2 + \log_2 m)\\
 & \leq & 4 L < \infty.
 \end{eqnarray*}
As previously Beppo Levy theorem implies that
 \begin{eqnarray*}
 \mathbb{E}\left(\sum_{k=0}^\infty (S^{\circ}_k)^2\right)  =  \sum_{k=0}^\infty (\mathbb{E}(S^{\circ}_k)^2) \leq 4 L < \infty.
 \end{eqnarray*}
Therefore 
 $\lim_{k \to \infty} S^{\circ}_k = 0$ almost surely. This together with the convergence of the sequence $(S_{2^k})$ implies that the series 
  $\sum_{n=1}^\infty a_n X_n$ converges almost surely. 
 Finally, since $$\sup_{m\in \mathbb{N}}|a_n X_1 + a_2 X_2 + \ldots + a_n X_n|^2 \leq (S^*)^2 + (S^\circ)^2 \mbox{ and }(S^\circ)^2 \leq \sum_{k=0}^\infty (S^\circ_k)^2,$$ it follows that 
 
 $$\mathbb{E}\left(\sup_{m\in \mathbb{N}}|a_n X_1 + a_2 X_2 + \ldots + a_n X_n|^2)\right) \leq  (\mathbb{E}(S^*)^2) + (\mathbb{E}(S^\circ)^2) \leq 8 L.$$

\section{Some important particular cases}

\begin{corollary} \label{th02} Let $(a_n)$ be a sequence of real or complex numbers and  $(X_n)$ be a sequence of real random variables such that $\mathbb{E}(X_n^2) = 1$ for all $n$  and  $\mathbb{E}(X_n X_m) \leq 0$ for all $n\ne m$. \\
 If~$\sum_{n=1}^\infty |a_n|^2 \log_2^2(n+1) <~\infty,$ then 
 the series $\sum_{n=1}^\infty a_n X_n$ converges almost surely.  
\end{corollary}

\begin{proof}
Let $c_n = a_n \log_2(n+1)$ for all $n$. 
Since  $\mathbb{E}(X_n X_m) \leq 0$ for $n \ne m$, then
$$\sum_{n \ne m} |c_n| |c_m| |\mathbb{E}(X_n X_m)|   =  - \sum_{n \ne m} |c_n| |c_m| \mathbb{E}(X_n X_m).$$ Hence
$$ \sum_{n,m =1}^\infty |c_n| |c_m| |\mathbb{E}(X_n X_m)| = \sum_{n=1}^\infty c_n^2  -  \sum_{n \ne m} |c_n| |c_m| \mathbb{E}(X_n X_m) \geq 0.$$ This yields
$$\sum_{n\ne m} |c_n| |c_m| |\mathbb{E}(X_n X_m)|  \leq  \sum_{n=1}^\infty c_n^2$$ and hence 
 $$\sum_{n,m =1}^\infty |c_n| |c_m| |\mathbb{E}(X_n X_m)| \leq 2 \sum_{n=1}^\infty c_n^2$$ and the conclusion follows by the theorem. 

\end{proof}
Corollary \ref{th02}  generalises previous results of {Matu\l{a}} \cite{Matula} and Antonini,  Kozachenko and  Volodin  \cite{Antonini} on the a.s. convergence of series of linear combinations of negatively correlated random variables. It indicates that with respect to a.s. convergence, sequences of real random variables that are negatively correlated  behave like sequences of uncorrelated random variables.

\begin{corollary}
Let $(a_n)$ be a sequence of real or complex numbers and $(X_n)$ be a  sequence of real or complex random variables with covariance matrix $\gamma(n,m))$.   Assume that there exists $b \geq 0$ such that the linear operator defined by the matrix $(\beta_{n,m})$ with 
$$\beta_{n,m} = |\gamma(n,m)| n^{-b} m^{-b}$$ in $\ell^{2}(\mathbb{N})$ is bounded. 
If \begin{eqnarray} \label{VK1}
\sum_{n=1}^\infty |a_n|^2 n^{2b} \log_n^2(n+1) < \infty,
\end{eqnarray}
then the series $\sum_{n=1}^\infty a_n X_n$ converges almost surely. 

\end{corollary}
In particular if  $b = 0$, that is, the covariance matrix $(\gamma(n,m))$ defines a bounded operator, then we retrieve that $\sum_{n=1}^\infty |a_n|^2 \log_n^2(n+1) < \infty$ implies the almost surely convergence of $\sum_{n=1}^\infty a_n X_n$. This is an extension of the Menshov--Rademacher theorem from sequences of orthogonal functions to sequences of functions $(\varphi_n(x))$ such that the ``covariance matrix''
 $((\gamma(n,m))$ given by:
   $$\gamma(n,m) = \int_{\mathbb{T}} \varphi_n(x) \overline{\varphi_m(x)} dx$$
defines a bounded operator in $\ell^2.$

\begin{proof}
Set $c_n = a_n \log_2(n+1)$ so that condition (\ref{VK1}) becomes $\sum_{n=1}^\infty |c_n|^2 n^{2b} < \infty $. 
Write
      \begin{eqnarray*}
\sum_{n, m=1}^\infty |c_n| |c_m| |\mathbb{E}(X_{n} \overline{X_{m}})| =  \sum_{n,m =1}^\infty |c_n| |c_m| |\gamma(n,m)| =  \sum_{n,m =1}^\infty |c_n| |c_m| n^b m^b \beta_{nm}. 
 \end{eqnarray*}
Now the boundedness of $(\beta_{n,m})$ implies that for some fixed $K >0$, 
 $$\sum_{n, m=1}^\infty |c_n| |c_m| |\mathbb{E}(X_{n} \overline{X_{m}})| \leq K  \sum_{n,m =1}^\infty |c_n|^2 n^{2b} < \infty.$$
That is,
 $$\sum_{n, m=1}^\infty |a_n| |a_m| |\mathbb{E}(X_{n} \overline{X_{m}})| \log_2(n+1) \log_2(m+1) < \infty$$ and the conclusion follows. 
\end{proof}

\begin{corollary} \label{Ferdinand}
Assume that the covariance matrix is such that $|\gamma(n,m)| \leq K|n-m|^{-a}$ for some fixed constants $K>0$ and $a \geq  0$.  Then if there exists $b > (1-a)/2$, such that  $$\sum_{n=1}^\infty |a_n|^2 n^{2b} \log_n^2(n+1) < \infty,$$ then the series $\sum_{n=1}^\infty a_n X_n$ converges almost surely. 
\end{corollary}
It is so because one can show that the matrix $(\beta_{n,m})$ given by $\beta_{n,m} = |n-m|^{-a} n^{-b} m^{-b}$ is bounded. (This can be obtained by using the Schur test to the vector $x = (x_n)$ with $x_n = n^{-c}$  for some constant $c \geq 0$ such that $a + b + c >1$.)

\section{Application to a.e. convergence of trigonometric series}
In particular consider a Borel probability measure $\mu$ on $\mathbb{T}$ and the trigonometric sequence $(e_n)$ given by 
$e_n(x) = \exp(2 \pi i n x)$, $(n \in \mathbb{Z})$.  
Clearly, $$\langle e_n, e_m\rangle = \int_{\mathbb{T}} e_n(x) \overline{e_m(x)} d\mu(x) = \hat \mu(n-m).$$
Then if 
      $$\sum_{n,m=-\infty}^\infty |a_n| |a_m| |\hat \mu(n-m)| \log_2(|n|+1) \log_2(|m|+1) < \infty,$$ 
then the trigonometric  series $\sum_{n=-infty}^\infty a_n \exp(2 \pi i n x) $ converges $\mu$-a.e. in $\mathbb{T}$. Then Corollary \ref{Ferdinand} immediately implies the following: 

\begin{corollary}
Assume that for constant $K >0$ and $a \geq 0$, 
\begin{eqnarray} \label{kaziba_enyanya}
|\hat \mu(n)| \leq K |n|^{-a} \mbox{ for all } n \neq 0.
\end{eqnarray}
If there exists $b>(1-a)/2$ such that ~$\sum_{n \in \mathbb{Z}} |a_n|^2 |n|^{2b} \log_2^2(|n|+1) <~\infty,$ then the trigonometric series $\sum_{n\in \mathbb{Z}} a_n \exp(2\pi i n t)$ converges $\mu$-almost everywhere on $\mathbb{T}$.
\end{corollary}
Consider now a compact subset $E$ of $\mathbb{T}$ of Lebesgue measure zero which supports a probability measure $\mu$ such that inequality (\ref{kaziba_enyanya}) holds. Let $f \in L^2(\mathbb{T})$ such that $\sum_{n \in \mathbb{Z}} |\hat f(n)|^2 |n|^{2b} \log_2^2(n+1) < \infty$ for some $b > (1-a)/2$. 
Then the Fourier series $\sum_{n \in \mathbb{Z}} \hat f(n) \exp(2\pi n it)$ of $f$ converges $\mu$-a.e. on $E$. In particular if $f$ is in the Sobolev space $H^{p}(\mathbb{T})$ for some $p > (1-a)/2$ with $0 \leq a \leq 1$ (that is, $\sum_{n\in \mathbb{Z}} |\hat f(n)|^2 (1 + n^2)^{p} < \infty$), then we can choose $b$ with  $(1-a)/2 < b < p$ so that $\sum_{n \in \mathbb{Z}} |\hat f(n)|^2 |n|^{2b} \log_2^2(n+1) < \infty$ and obtain that the Fourier series of $f$ converges $\mu$-a.e. on $E$.

We recall that for a Lebesgue measurable subset $E$ of $\mathbb{T}$, the Fourier dimension of $E$ is the supremum of the numbers $0 \leq \alpha \leq 1$ such that $E$ supports a Borel probability measure $\mu$ such that
 $|\hat \mu(u)|^2 = o(|u|^{-\alpha})$ for $u \to \infty$. It is well known that the Fourier dimension is always less than or equal to the  Hausdorff dimension and the two can be different (see e.g. Kahane \cite[p. 250]{Kahane_1985} and Mattila \cite[p. 40]{Mattila}). 

A classical result by Ketznelson and Kahane says that if $E$ is a subset of $\mathbb{T}$ of Lebesgue measure $0$, then there exists a continuous function $f$ on $\mathbb{T}$ such that its Fourier series diverges everywhere on $E$. But now our observation here is that if $E$ has positive Fourier dimension $\alpha$, $0 < \alpha \leq 1$, such function must be such that 
$$\sum_{n=-\infty}^\infty |\hat f(n)|^2 |n|^{2b} \log_2^2(|n|+1) = \infty\,\,\, \mbox{ for all } b > (1-\alpha)/2.$$
In other words the Fourier series of a well-behaved $L^2$ function $f$ in the sense that\\ $\sum_{n=1}^\infty |\hat f(n)|^2 n^{2b} \log_2^2(n+1) < \infty$ for some number $0 \leq b \leq 1/2$ cannot diverge everywhere on a subset of Fourier dimension $> 1 - 2b$. Or again, the Fourier series of a function $f\in H^{p}(\mathbb{T})$ converges $\mu$-a.e. for every Borel probability measure  $\mu$ such that $|\hat \mu(n)|^2 \leq  K |n|^{-\beta}$ for some $\beta > 1- 2p$.

\section{Remark for Gaussian random series}
In the particular case of Gaussian random variables, we obtain that we can get rid of the factor $\log_2^2(1+n)$ thanks to the Sudakov--Fernique lemma. We have the following: 
\begin{theorem}
Let $(a_n)$ be a sequence of real or complex numbers and $(X_n)$ be a sequence of real or complex Gaussian identically distributed random variables with zero mean.  If
\begin{eqnarray} \label{Kaziba_enyanya}
\sum_{n,m=1}^\infty |a_n| |a_m| |\mathbb{E}( X_n  \overline{X_m})| < \infty, \end{eqnarray}
 then the series $\sum_{n=1}^\infty X_n$ converges almost surely.   
\end{theorem}
\begin{proof}
We shall  first prove the following inequality:
 \begin{eqnarray} \label{smewqweq234}
\mathbb{E}\left(\sup_{1\leq n \leq N}|X_1 + X_2 + \ldots+ X_n|\right) \leq 2 \left(\sum_{n, m = 1}^N |\mathbb{E}( X_n\overline{X_m})|\right)^{1/2}.
\end{eqnarray} 
Let $Z_1, Z_2, \ldots, Z_N$ be i.i.d Gaussian random variables with zero mean and unit variance. 
Consider the sequence of random variables $(Y_n)$, $1 \leq n \leq N$, given by 
           $$Y_n = \left(\sum_{j=1}^N |\mathbb{E}(X_n \overline{X_j})|\right)^{1/2} Z_n.$$
Clearly for all $1 \leq n \leq m \leq N,$
 $$\mathbb{E}(|X_{n} + X_{n+1} + \ldots + X_m|^2) \leq \mathbb{E}(|Y_{n} + Y_{n+1} + \ldots + Y_m|^2).$$
Then by the classical Sudakov-Fernique inequality 
      $$\mathbb{E}\left(\sup_{1\leq n \leq N}|X_1 + X_2 + \ldots+ X_n|\right) \leq \mathbb{E}\left(\sup_{1\leq n \leq N}|Y_1 + Y_2 + \ldots+ Y_n|\right).$$              
 Now since the variables $(Y_n)$ are independent, it is well-known that 
   $$\mathbb{E}\left(\sup_{1\leq n \leq N}|Y_1 + Y_2 + \ldots+ Y_n|\right) \leq 2 \mathbb{E}\left(|Y_1 + Y_2 + \ldots+ Y_N|\right).$$
   Since obviously $(\mathbb{E}|X|)^2\leq \mathbb{E}(|X|^2)$, this yields,
    $$\mathbb{E}\left(\sup_{1\leq n \leq N}|X_1 + X_2 + \ldots+ X_n|\right) \leq 2 \left(\mathbb{E}\left(|Y_1 + Y_2 + \ldots+ Y_N|^2\right)\right)^{1/2}.$$
Hence
      $$\mathbb{E}\left(\sup_{1\leq n \leq N}|X_1 + X_2 + \ldots+ X_n|\right) \leq 2 \left(\sum_{k,j=1}^N |\mathbb{E}(X_k \overline{X_j})|\right)^{1/2}.$$
Now fix a real number $r>0$ and  $m \in \mathbb{N}$. Then 
\begin{eqnarray*}
\mathbb{P}\left(\sup_{1 \leq j < \infty}|X_m + X_{m+1} + \ldots + X_{m+j}|^{1/2} > r\right) &\leq & \frac{1}{r^2} \mathbb{E}\left(\sup_{1 \leq j < \infty}|X_m + X_{m+1} + \ldots + X_{m+j}|\right)\\
 & \leq & \frac{2}{r^2} \left(\sum_{k,j=m}^\infty |\mathbb{E}(X_k \overline{X_j})|\right)^{1/2}.
\end{eqnarray*}
Since  $\sum_{k,j=1}^\infty |\mathbb{E}(X_k \overline{X_j})| < \infty$, then 
 $$\lim_{m\to \infty} \sum_{k,j=m}^\infty |\mathbb{E}(X_k \overline{X_j})| = 0.$$
Hence
$$\lim_{m\to \infty} \mathbb{P}\left(\sup_{1 \leq j < \infty}|X_m + X_{m+1} + \ldots + X_{m+j}|^{1/2} > r\right) = 0.$$
Since $r$ can be taken arbitrary small, this yields
 $$\lim_{m\to \infty} \mathbb{P}\left(\sup_{1 \leq j < \infty}|X_m + X_{m+1} + \ldots + X_{m+j}|^{1/2} > 0\right) = 0.$$
This implies by an application of Fatou's lemma (as in Kahane \cite[p 30]{Kahane_1985}) that the series $\sum_{n=1}^\infty X_n$ converges almost surely.   \end{proof}

\end{document}